\documentclass[12pt, reqno]{amsart}
\usepackage{amsmath, amsthm, amscd, amsfonts, amssymb, graphicx, xcolor}
\usepackage{mathtools}
\usepackage[bookmarksnumbered, colorlinks, plainpages]{hyperref}
\usepackage{enumitem}
\usepackage[english]{babel}
\usepackage{fancyhdr}

\newtheorem{theorem}{Theorem}[section]
\newtheorem{lemma}[theorem]{Lemma}

\theoremstyle{definition}
\newtheorem{definition}[theorem]{Definition}
\newtheorem{example}[theorem]{Example}

\theoremstyle{remark}

\numberwithin{equation}{section}

\begin{document}
\setcounter{page}{1}

\color{darkgray}{
\noindent 
{\small Seminar Nasional Pendidikan Matematika (SNPM)}\hfill     {\small ISSN: 2988-3458}\\
{\small Vol 1 (2023) 540-552}\hfill  {\small doi.org/10.31219/osf.io/7x3ks}}

\centerline{}

\centerline{}

\title[Structures of \((m,n)\)-seminearring]{Structures of \(\boldsymbol{(m,n)}\)-seminearring
}

\author[M. S. L. Liedokto]{Muhsang Sudadama Lieko Liedokto$^1$$^{*}$}

\address{$^{1}$ Department of Mathematics, Universitas Negeri Malang, Malang, Indonesia.}
\email{\textcolor[rgb]{0.00,0.00,0.84}{muhsangsll@gmail.com}}



\keywords{seminearring, \((m,n)\)-ring, \((m,n)\)-semiring, \((m,n)\)-seminearring
\newline \indent $^{*}$ Corresponding author}


\begin{abstract}
This article introduces the \((m,n)\)-seminearring structure, which is a generalization of \((m,n)\)-semiring. This research aims to develop theories of \((m,n)\)-seminearring. In particular, the concepts of \((m,n)\)-seminearring, \((m,n)\)-subseminearring, \((m,n)\)-ideal, \((m,n)\)-seminearring with unity, homomorphism of \((m,n)\)-seminearrings, construction of factor \((m,n)\)-seminearring of \((m,n)\)-seminearring by congruence relation, and some of its exciting properties are given. The method used in this study is to adopt the theory in \((m,n)\)-semiring.
\end{abstract} \maketitle

\section{Introduction}

\noindent Algebra structures play a vital role in mathematics, and they have a wide range of applications in many disciplines, such as theoretical physics, computer science, graph theory, cryptography, statistics, economics, and other fields. Research in generalizations of algebra structures has been a main focus of mathematicians for a long time. One interesting generalization is ternary algebraic systems theory, which was first introduced by Lehmer \cite{Lehmer}.

Lister \cite{Lister} introduced the ternary multiplication operation on rings called a ternary ring. A ternary ring is a special case of $ (m, n) $-ring. The $ (m, n) $-ring structure has been studied by researchers such as Crombez \cite{Crombez}; Crombez \& Timm \cite{Crombez-Timm}; Dudek \cite{Dudek}; Leeson \& Butson \cite{Leeson(a),Leeson(b)}; Pop \cite{Pop}; Pop \& Iancu \cite{Pop-Iancu(1),Pop-Iancu(2),Iancu-Pop}.Dutta \& Kar \cite{Dutta-Kar} introduced an idea of ternary semiring that is a generalization of a ternary ring. A ternary semiring is a special case of $ (m,n) $-semiring. Further, the concept of $ (m, n) $-semiring has been introduced, and some of its properties are discussed in the article: Pop \& Pop \cite{Pop-Pop(1),Pop-Pop(2)}; Alam et al. \cite{Alam}; Hila \cite{Hila-Kar-Kuka}; Pop \& Lauran \cite{Pop-Lauran(1)}; Davvaz \& Mohammadi \cite{Davvaz-Mohammadi}; Chaudhari et al. \cite{Chaudhari}. Subsequently, Vijayakumar \& Dhivya \cite{Vijayakumar-Bharathi(2)} introduced the concept of ternary seminearring, which is a generalization of ternary semiring. In their article, Vijayakumar \& Dhivya \cite{Vijayakumar-Bharathi(2)} described the concepts of ternary seminearring, an ideal of ternary seminearring, homomorphism of ternary seminearrings, and their properties. In the next article, Vijayakumar \& Dhivya \cite{Vijayakumar-Bharathi(1)} constructed factor ternary seminearring by congruence relation on ternary seminearring and proved the isomorphism theorem on ternary seminearring.

However, to the researcher's knowledge, there has been no research related to \((m,n)\)-seminearring, that is, the generalization of ternary seminearring. The concept of \((m,n)\)-seminearring is intriguing and has great potential to provide a deeper understanding of algebraic properties in a more general context. Therefore, this research aims to develop some theory related to \((m,n)\)-seminearring. The focus of this research will be on understanding the concepts of \((m,n)\)-semi- nearring, \((m,n)\)-subseminearring, \((m,n)\)-ideal, \((m,n)\)-seminearring with unity, homomorphism of \((m,n)\)-seminearrings, and construction of factor \((m,n)\)-semi- nearring of \((m,n)\)-seminearring by congruence relation. This research is expected to significantly contribute to broadening the understanding of more general algebraic structures and exploring interesting properties of $ (m,n) $-seminearrings.

\section{METHODS}
\noindent Systematic steps will be taken while working on this article. First, an extensive literature review is conducted to review previous contributions and identify research gaps. Second, understanding the properties of ternary seminearring and $ (m,n)$-semiring structures. Finally, definitions and properties are formulated by adopting existing theories in $ (m, n) $-semiring.

\section{RESULTS AND DISCUSSION}
Let \(X\) be a nonempty set. A subclass \(\rho\) of \(X\times X\) is called a relation on \(X\). A relation \(\rho\) on \(X\) is an equivalence relation on \(X\) provided \(\rho\) is: reflexive, \((x,x)\in \rho\) for all \(x\in X\); symmetric, if \((x,y)\in \rho\) then \((y,x)\in \rho\); transitive, if \((x,y)\in \rho\) and \((y,z)\in X\) then \((x,z)\in \rho\). If \(\rho\) is an equivalence relation on \(X\) and \((x,y)\in \rho\), we say that \(x\) is equivalent to \(y\) under \(\rho\) and write \(x\rho y\). Let \(\rho\) be an equivalence relation on \(X\). If \(x\in X\), the equivalence class of \(x\)
(denoted \(\rho(x)\)) is the class of all those elements of \(X\) that are equivalent to \(x\); that is, \(\{y\in X \mid y\rho x\}\). The class of all equivalence classes in \(X\) is denoted \(X/\rho\) and called the quotient class of \(X\) by $\rho$. In this context, it can be proved that for all \(x,y\in X\), several important properties hold: \(x\in \rho\); if \(x\in \rho(y)\) then \(\rho(x)=\rho(y)\); \(\rho(x)=\rho(y)\) if and only if \(x\rho y\); either \(\rho(x)=\rho(y)\) or \(\rho(x)\cap\rho(y)=\emptyset\) \cite[p. 7]{Hungerford}.

Let \(R\) be a nonempty set and \(f\) be a mapping \(f:R^m\to R\), i.e., \(f\) is an \(m\)-ary
operation. Elements of the set \(R\) are denoted by \(a_i,b_i\) where \(i\in \mathbb{Z}^{+}\). A nonempty set \(R\) together with an \(m\)-ary operation is called an \( m \)-groupoid, and it is denoted by \((R,f)\). The following general convention is applied: the sequence \(a_i,a_{i+1},\ldots,a_j\) is denoted by \(a_i^j\); if \(j<i\), then \(a_i^j\) is the empty symbol. The notation \(\underbrace{a,a,\dotsc ,a}_{k}\) is abbreviated as \(a^k\).  Hence, the sequence \(a_1,a_2,\ldots,a_m\), can be expressed \(a_1^m\). For all \(1\leq i < j \leq m\) the term \(f(a_1,a_2,\ldots,a_j,b_{i+1},b_{i+2},\ldots,b_j,c_{j+1},\) \(c_{j+2},\ldots,c_m ) \) is represented as \(f\left(a_1^i,b_{i+1}^j,c_{j+1}^m \right) \). Specifically, if \(b_{i+1}=b_{i+2}=\cdots=b_j=b\), this expression simplifies to \(f\left(a_1^j,b^{j-i},c_{j+1}^m \right) \).
\begin{definition}
	\cite[Definition 1.1]{Chaudhari} Let \(R\) be a nonempty set, \( f:R^{m}\rightarrow R\) is an \( m\)-ary operation. Then
	\begin{enumerate}
		\item Associative law for \( f\) is defined as
		\[
			f\left( a_{1}^{i-1} ,\left( a_{i}^{m+i-1}\right) ,a_{m+i}^{2m-1}\right) =f\left( a_{1}^{j-1} ,f\left( a_{j}^{m+j-1}\right) ,a_{m+j}^{2m-1}\right)
		\]
		for all \( a_{1}^{2m-1} \in R\) and \( 1\leqslant i< j\leqslant m\).
	\item Commutative law for \( f\) is defined as
	\begin{equation*}
		f\left( a_{1}^{m}\right) =f\left( a_{\eta ( 1)} ,a_{\eta ( 2)} ,\dotsc ,a_{\eta ( m)}\right) 
	\end{equation*}
	for every permutation \( \eta \) of \( \{1,2,\dotsc ,m\}\).
	\item An \(m\)-groupoid \((R,f)\) is called \(m\)-semigroup if \(f\) is associative.
	\item An \(m\)-semigroup \((R,f)\) is called \(m\)-commutative semigroup if \(f\) is commutative.
	\item An \( n \)-ary operation \( g:R^{n}\rightarrow R\) is said to be distributive with respect to the \( m \)-ary operation \( f \) if \begin{equation*}
		g\left( b_{1}^{i-1} ,f\left( a_{1}^{m}\right) ,b_{i+1}^{n}\right) =f\left( g\left( b_{1}^{i-1} ,a_{1} ,b{_{i+1}^{n}}\right),\dotsc ,g\left( b_{1}^{i-1} ,a_{m} ,b{_{i+1}^{n}}\right)\right)\end{equation*}
	for all \(a_{1}^{m} ,b_{1}^{i-1} ,b_{i+1}^{n}\in R\) and \(1\leqslant i\leqslant n\).
\end{enumerate}
\end{definition}

\begin{definition} \cite[Definition 1.2]{Chaudhari}
	Let \( R\) be a nonempty set together with \( m \)-ary operation \( f:R^{m}\rightarrow R\) called addition and \( n \)-ary operation \( g:R^{n}\rightarrow R\) called multiplication. The triple \((R,f,g)\) is called \((m,n)\)-semiring if it satisfies the following conditions:
	\begin{enumerate}
		\item \((R,f)\) is an \(m\)-commutative semigroup;
		\item \((R,g)\) is an \(n\)-semigroup;
		\item there exists \(0\in R\) (called zero element of \(R\)) such that
		\begin{enumerate}[label=(\roman*)]
			\item 0 is \(f\)-identity of \(R\) (i.e. \(f\left(a,0^{m-1} \right)=a  \) for every \(a\in R\)),
			\item 0 is \(g\)-zero of \(R\) (i.e. if \(a_1^n\in R\) and \(a_i=0\) for some \(i\) then  \(g\left(a_1^n \right)=0\));
		\end{enumerate}
	\item $ g $ is distributive with respect to $ f $.
	\end{enumerate}
\end{definition}
\subsection{\(\boldsymbol{(m,n)}\)-seminearring} This section introduces the concept of \( (m, n) \)-semi- nearring and the properties obtained.
\begin{definition}
	Let \( R\) be a nonempty set, \( f:R^{m}\rightarrow R\) is an \( m \)-ary operation, and \( g:R^{n}\rightarrow R\) is an \( n \)-ary operation. The \( n\)-ary operation \( g\) is called \( t\)-distributive with respect to \( f\) if
	\begin{equation*}
		g\left( b_{1}^{t-1} ,f\left( a_{1}^{m}\right) ,b_{t+1}^{n}\right) =f\left( g\left( b_{1}^{t-1} ,a_{1} ,b{_{t+1}^{n}}\right) ,\dotsc ,g\left( b_{1}^{t-1} ,a_{m} ,b{_{t+1}^{n}}\right)\right) 
	\end{equation*}
	for all \( a_{1}^{m} ,b_{1}^{t-1} ,b_{t+1}^{n}\in R\). If \( g\) is 1-distributive, then \( g\) is called right distributive with respect to \( f\). If \( g\) is \( n\)-distributive with respect to \( f\), then \( g\) is called distributive with respect to \( f\). Furthermore, if \( g\) is \( t\)-distributive with respect to \( f\) for every \( 1\leqslant t\leqslant n\), then \( g\) is distributive with respect to \( f\). 
\end{definition}
\begin{definition}
		Let \( R\) be a nonempty set together with \( m \)-ary operation \( f:R^{m}\rightarrow R\) called addition and \( n \)-ary operation \( g:R^{n}\rightarrow R\) called multiplication. The triple \((R,f,g)\) is called \(t\)-\((m,n)\)-seminearring if it satisfies the following conditions:
		\begin{enumerate}
			\item \((R,f)\) is an \(m\)-semigroup;
			\item \((R,g)\) is an \(n\)-semigroup;
			\item $ g $ is \(t\)-distributive with respect to $ f $.
		\end{enumerate}
	The triple \( (R, f, g) \) is called right \( (m,n) \)-seminearring if \( (R, f, g) \) is an 1-\( (m,n) \)-seminearring and \( (R, f, g) \) is left \( (m,n) \)-seminearring if \( (R, f, g) \) is an \( n \)-\( (m,n) \)-seminearring.
\end{definition}

\begin{example}\mbox{}\par\nobreak
	\begin{enumerate}
		\item Let \( A\) be a nonempty set and \( 2^{A}\coloneqq\left\lbrace S \middle| S \subseteq A\right\rbrace \). The triple \(\left( 2^{A} ,f,g\right)\) is \( t\)-\( ( m,n)\)-seminearring together with \( m \)-ary operation \( f\left( A_{1}^{m}\right) =\bigcup_{i=1}^{m} A_{i}\) and \( n \)-ary operation \( g\left( B_{1}^{n}\right) =\bigcap_{i=1}^{n} B_{i}\) for all \( A_{1}^{m} ,B_{1}^{m} \in 2^{A}\).
		
		\item Suppose \((R,+,\cdotp)\) is a ring. The triple \((R,f,g)\) is \( t\)-\((m, n)\)-seminearring together with \( m \)-ary operation \( f\left( a_{1}^{m}\right) =\sum_{i=1}^{m} a_{i}\) and \( n \)-ary operation \( g\left( b_{1}^{n}\right) =\prod _{i=1}^{n} b_{i}\) for all \( a_{1}^{m} , b_{1}^{n} \in R\).
		
		\item The \( t\)-\( (m,n)\)-seminearring is not necessarily \( (m,n)\)-semiring. Let us denote by \(T\) the set $ \left\lbrace( a,b)\middle| a,b\in \mathbb{R}\right\rbrace  $. The triple \((T,f,g)\) is \( t\)-\((2,3)\)-seminear- ring together with 2-ary operation \( f(( a_{1} ,b_{1}) ,( a_{2} ,b_{2})) =( a_{1} +a_{2} ,b_{1} +b_{2})\) and 3-ary operation \( g((a_{1} ,b_{1}) ,( a_{2} ,b_{2}) ,( a_{3} ,b_{3})) =( a_{1} a_{2} a_{3} ,b_{1} a_{2} a_{3} +b_{2} a_{3} +b_{3})\), but is not \(( m,n)\)-semiring because \( g\) is not right distributive with respect to \( f\) \cite{Vijayakumar-Bharathi(1)}.
	\end{enumerate}
\end{example}
\begin{definition}
	An \( t\)-\( ( m,n)\)-seminearring \( ( R,f,g)\) is an \( t\)-\( ( m,n)\)-seminearring with absorbing zero if there exists \( 0\in R\) such that 
	\begin{enumerate}
		\item \( 0\) is  \( f\)-identity of \( R\) (i.e. \( f\left( 0^{i-1} ,a,0^{m-i}\right)=a\) for all \( a\in R\) and \( 1\leqslant i\leqslant m\))
		\item 0 is \( g\)-zero (i.e. \( g\left( a_{2}^{i} ,0,a_{i+1}^{n}\right)=0\) for all \( a_{2}^{n}\in R\) and \( 1\leqslant i\leqslant n\))
	\end{enumerate}
\end{definition}

From now on, unless specified otherwise, "\( (m,n) \)-seminearring" means "\(n\)-\( (m,n) \)-seminearring" and it is understood that all theorems about \(n\)-\( (m,n) \)-seminearring also hold, mutatis mutandis, for \( t \)-\( (m,n) \)-seminearring for all \(1\leqslant t < n\).
 
\subsection{\(\boldsymbol{(m,n)}\)-subseminearring} This section introduces the concept of \( (m, n) \)-subseminearring and the properties obtained.
\begin{definition}
	Let $ (R, f, g) $ be an $ (m, n) $-seminearring and $ S $ is a nonempty subset of $ R $. The triple $ (S, f, g) $ is called an $ (m, n) $-subseminearring of $ (R, f, g) $ if $ (S, f, g) $ is an $ (m, n) $-seminearring. 
\end{definition}
\begin{lemma}\label{lemma-subseminearring}
	Let \((R,f,g)\) be an \((m,n)\)-seminearring, \(S\subseteq R\), and \(S\neq \emptyset\). Then \((S,f,g)\) is an \((m,n)\)-subseminearring of \((R,f,g)\) if and only if the following two conditions are 
	satisfied. 
	\begin{enumerate}
		\item If \(a_1^m\in S\) then \(f\left( a_1^m\right) \in S\).
		\item If \(b_1^n\in S\) then \(g\left( b_1^n\right) \in S\).
	\end{enumerate}
\end{lemma}
\begin{proof}
	\(( \Rightarrow )\) They are easily seen to be smooth. \( ( \Leftarrow )\) By part (1) and (2), \( S\) is closed under \( f\) and \( g\). Since \( S\subseteq R\) and \( ( R,f,g)\) is an \( ( m,n)\)-seminearring, \( f\) and \( g\) are associative, and \( g\) is left distributive with respect to \( f\). Thus, \( ( S,f,g)\) is an \( ( m,n)\)-seminearring. Therefore, \( ( S,f,g)\) is an \( ( m,n)\)-subseminearring of \( ( R,f,g)\).
\end{proof}
\begin{theorem}
	Let \(\Delta\) be an index set, \( ( R,f,g)\) is an \( ( m,n)\)-seminearring, and \( \{( S_{\alpha } ,f,g) | \alpha \in \Delta \}\) is a collection of \( ( m,n)\)-subseminearrings of \( ( R,f,g)\) where \( S\coloneqq\bigcap _{\alpha \in \Delta } S_{\alpha } \neq \emptyset \). Then \( \left(S,f,g\right)\) is an \( ( m,n)\)-subseminearring of \( ( R,f,g)\).
\end{theorem}
\begin{proof}
	Suppose that  \( a_{1}^{m} ,b_{1}^{n} \in S\). Then \( a_{1}^{m} ,b_{1}^{n} \in S_{\alpha }\) for all \( \alpha \in \Delta\). Since \( S_{\alpha }\) is closed under \( f\) and \( g\), \( f\left( a_{1}^{m}\right) ,g\left( b_{1}^{n}\right) \in S_{\alpha }\) for all \( \alpha \in \Delta\). Hence, \( f\left( a_{1}^{m}\right) ,g\left( b_{1}^{n}\right) \in S\). By Lemma \ref{lemma-subseminearring}, \( \left(S ,f,g\right)\) is an \((m,n)\)-subseminearring of \((R,f,g)\).
\end{proof}
\subsection{\(\boldsymbol{(m,n)}\)-ideal} This section will describe the \((m, n)\)-ideal of \( (m, n) \)-seminearring and some properties obtained.
\begin{definition}
Let \(( R,f,g)\) be an \(( m,n)\)-seminearring. An \(   ( m,n)\)-subseminearring \(   ( I,f,g)\) of an \(   ( m,n)\)-seminearring \(   ( R,f,g)\) is an \(   i\)-\(   ( m,n)\)-ideal provided 
\begin{equation*}
	b_{1}^{i-1} ,b_{i+1}^{n} \in R\ \ \  \text{and} \ \ \ \  x\in I \ \ \ \ \ \Longrightarrow \ \ \ \ \ \ g\left( b_{1}^{i-1} ,x,b_{i+1}^{n}\right) \in I
\end{equation*}
The triple \(   ( I,f,g)\) is called a left \(   ( m,n)\)-ideal if \(   ( I,f,g)\) is \(   1\)-\(   ( m,n)\)-ideal. The triple \(   ( I,f,g)\) is called a right \(   ( m,n)\)-ideal if \(   ( I,f,g)\) is an \(   n\)-\(   ( m,n)\)-ideal. Furthermore, \(   ( I,f,g)\) is called \(   ( m,n)\)-ideal if \(   ( I,f,g)\) is \(   i\)-\(   ( m,n)\)-ideal for every \(   1\leqslant i\leqslant n\).
\end{definition}
\begin{lemma}\label{lemma-ideals}
	Let \( (R,f,g)\) be an \((m,n)\)-seminearring, \( I\subseteq R\), and \( I\neq \emptyset\). The triple \((I,f,g)\) is an \((m,n)\)-ideal of \((R,f,g)\) if and only if
	\begin{enumerate}
		\item \( f\left( a_{1}^{m}\right) \in I\) for all \( a_{1}^{m}\in I\) and
		\item \( g\left( b_{1}^{i-1} ,x,b_{i+1}^{n}\right) \in I\) for all \( b_{1}^{i-1} ,b_{i+1}^{n} \in R\), \( x\in I\), and \( 1\leqslant i\leqslant n\).
	\end{enumerate}
\end{lemma}
\begin{proof}
	\(( \Rightarrow )\) The proof is standard. \(( \Leftarrow )\). Let \( a_{1}^{m} ,b_{1}^{n} \in I\). By part (1), \( f\left( a_{1}^{m}\right) \in I\). Since \( b_{1}^{i-1} ,b_{i+1}^{n}\in R\) and \( b_{i}\in I\) for some \( i\), by part (2), \( g\left( b_{1}^{n}\right) =g\left( b_{1}^{i-1} ,b_{i} ,b_{i+1}^{n}\right) \in I\). Thus, \((I,f,g)\) is an \((m,n)\)-subseminearring of \((R,f,g)\). 
\end{proof}
\begin{theorem}
	Let \(\Delta\) be an index set, \(( R,f,g)\) is an \(( m,n)\)-seminearring, and \(\left\lbrace ( I_{\alpha } ,f,g) \middle| \alpha \in \Delta \right\rbrace \) is a collection of \(( m,n)\)-ideals of \(( R,f,g)\) where \(I\coloneqq \bigcap_{\alpha \in \Delta } I_{\alpha } \neq \emptyset\). Then \(\left( I,f,g\right)\) is an \(( m,n)\)-ideal of \(( R,f,g)\).
\end{theorem}
\begin{proof}
	If \( a_{1}^{m} ,x\in I\) and \( b_{1}^{i-1} ,b_{i+1}^{n} \in R\) \((1\leqslant i\leqslant n)\) then \( a_{1}^{m} , x\in I_{\alpha }\) for all \( \alpha \in \Delta\) so that \( f\left( a_{1}^{m}\right) \in I_{\alpha }\) and \( g\left( b_{1}^{i-1},x,b_{i+1}^{n}\right) \in I_{\alpha }\) for all \( \alpha \in \Delta\). Hence  \( \)\( f\left( a_{1}^{m}\right) ,g\left( b_{1}^{i-1} ,x,b_{i+1}^{n}\right) \in I\) and \( \left(I ,f,g\right)\) is an \((m,n)\)-ideal of \( (R,f,g)\) by Lemma \ref{lemma-ideals}.
\end{proof}

\subsection{\(\boldsymbol{(m,n)}\)-seminearring with unity} This section will give some exciting properties of \( (m, n) \)-seminearring with unity.
\begin{definition}
	An \( (m,n) \)-seminearring \( (R,f,g) \) is called an \( (m,n) \)-seminearring with unity (or \( g \)-identity) if there exists \(1\in R\) such that 
	\begin{equation*}
		g\left( 1^{i-1} ,a,1^{n-i}\right) =a
	\end{equation*}
	for all \(a\in R\) and \(1 \leqslant i \leqslant n\).
\end{definition}

\begin{theorem}
	Let \((R,f,g)\) be an \((m,n)\)-seminearring with unity. Then \[g\left(a_2^i,1,a_{i+1}^n \right)=g\left(a_2^j,1,a_{j+1}^n \right)  \] for all \(a_2^n\in R\) and \(1\leqslant i < j \leqslant n\).
\end{theorem}
\begin{proof}
	Suppose that  \(a_2^n\in R\) and \(1\leqslant i < j \leqslant n\). Then
	\begin{align*}
		g\left( a_{2}^{ i} ,1,a_{ i+1}^{n}\right) & =g\left( a_{2}^{ i} ,1,a_{ i+1}^{ j} ,a_{ j+1}^{n}\right)\\
		& =g\left( a_{2}^{ i} ,1,g\left( a_{ i+1} ,1^{n-1}\right) ,g\left( a_{ i+2} ,1^{n-1}\right) ,\dotsc ,g\left( a_{ j} ,1^{n-1}\right) ,a_{ j+1}^{n}\right)\\
		& =g\left( a_{2}^{ i} ,g\left( 1,a_{ i+1} ,1^{n-2}\right) ,g\left( 1,a_{ i+2} ,1^{n-2}\right) ,\dotsc ,g\left( 1,a_{ j} ,1^{n-2}\right) ,1,a_{ j+1}^{n}\right)\\
		& =g\left( a_{2}^{ i} ,a_{ i+1} ,a_{ i+2} ,\dotsc ,a_{ j} ,1,a_{ j+1}^{n}\right)\\
		& =g\left( a_{2}^{ j} ,1,a_{ j+1}^{n}\right),
	\end{align*}
and the proof is complete. 
\end{proof}
\begin{theorem}
	Let \((R,f,g)\) be an \((m,n)\)-seminearring with unity \(1\in R\). If \((I,f,g)\) is an \((m,n)\)-ideal of \((R,f,g)\) and \(1\in I\), then \(I=R\).
\end{theorem}
\begin{proof}
	Clearly, $I \subseteq R$. Suppose \(a\in R\) and \(1\in I\). Then \(a=g\left(a,1^{n-1} \right)\in I \), so that \(R\subseteq I\).
\end{proof}
\begin{definition}
	Let \((R,f,g)\) be an \((m,n)\)-seminearring with unity. An element \(b\in R\) said to be invertible if there exists \(a\in R\) such that  \[g\left(a,b,1^{n-2} \right)=g\left(b,a,1^{n-2} \right)=1.  \]
	The element \( b\) is called \( g\)-inverse of \( a\) and is denoted \( a^{-1}\). The set of all invertible elements in \( R \) is denoted by \(\mathcal{I}(R) \).
\end{definition}
\begin{theorem}\label{xx-1}
	Let \((R,f,g)\) be an \((m,n)\)-seminearring with unity. If \( x\in \mathcal{I}(R)\), then 
	\[
	g\left( x,x^{-1},a_{3}^{n}\right) =g\left( x^{-1},x,a_{3}^{n}\right) =g\left( a_{3}^{n} ,x,x^{-1}\right) =g\left( a_{3}^{n} ,x^{-1},x\right)=g\left( a_{3}^{n} ,1^{2}\right)
	\]
	for all \( a_{3}^{n} \in R\).
\end{theorem}
\begin{proof}
	Suppose that  \(a_3^n\in R\) and \( x\in \mathcal{I}(R)\). Then\begin{align*}
		g\left( x,x^{-1} ,a_{3}^{n}\right) & =g\left( x,g\left( x^{-1} ,1^{n-1}\right) ,a_{3}^{n}\right)\\
		& =g\left( g\left( x,x^{-1} ,1^{n-2}\right) ,1,a_{3}^{n}\right)\\
		& =g\left( g\left( x^{-1} ,x,1^{n-2}\right) ,1,a_{3}^{n}\right)\\
		& =g\left( x^{-1} ,g\left( x,1^{n-1}\right) ,a_{3}^{n}\right)\\
		& =g\left( x^{-1} ,x,a_{3}^{n}\right)
	\end{align*}
and
\begin{align*}
	g\left( x,x^{-1} ,a_{3}^{n}\right) & =g\left( x,g\left( x^{-1} ,1^{n-1}\right) ,a_{3}^{n}\right)\\
	& =g\left( g\left( x,x^{-1} ,1^{n-2}\right) ,1,a_{3}^{n}\right)\\
	& =g\left( 1,1,a_{3}^{n}\right)\\
	& =g\left( a_{3}^{n} ,1,1\right)\\
	& =g\left( a_{3}^{n} ,g\left( x,x^{-1} ,1^{n-2}\right) ,1\right)\\
	& =g\left( a_{3}^{n} ,x,g\left( x^{-1} ,1^{n-1}\right)\right)\\
	& =g\left( a_{3}^{n} ,x,x^{-1}\right).
\end{align*}
Similarly, one shows that \(g\left( a_{3}^{n} ,x,x^{-1}\right) =g\left( a_{3}^{n} ,x^{-1} ,x\right)\). Furthermore, note that
\begin{align*}
	g\left( a_{3}^{n} ,x^{-1} ,x\right) & =g\left( a_{3}^{n} ,x^{-1} ,g\left( x,1^{n-1}\right)\right)\\
	& =g\left( a_{3}^{n} ,g\left( x^{-1} ,x,1^{n-2}\right) ,1\right)\\
	& =g\left( a_{3}^{n} ,1^{2}\right).
\end{align*}
This completes the proof of Theorem \ref{xx-1}.
\end{proof}

\begin{theorem}
	Let \((R,f,g)\) be an \((m,n)\)-seminearring with unity. If \(a_1^n\in \mathcal{I}(R)\), then \(g\left(a_1^n \right)\in \mathcal{I}(R) \). 
\end{theorem}
\begin{proof}
	Suppose \(a_{1}^{n}\in \mathcal{I}( R)\) is arbitrary. Note that
	\begin{align*}
		g\Bigl( g\left( a_{1}^{n}\right) , & g\left( a_{n}^{-1} ,a_{n-1}^{-1} ,\dotsc ,a_{1}^{-1}\right) ,1^{n-2}\Bigr)\\
		& =g\left( a_{1}^{n-1} ,g\left( g\left( a_{n} ,a_{n}^{-1} ,1^{n-2}\right) ,a_{n-1}^{-1} ,a_{n-2}^{-1} ,\dotsc ,a_{1}^{-1}\right)\right)\\
		& =g\left( a_{1}^{n-1} ,g\left( 1,a_{n-1}^{-1} ,a_{n-2}^{-1} ,\dotsc ,a_{1}^{-1}\right)\right)\\
		& =g\left( a_{1}^{n-2} ,1,g\left( g\left( a_{n-1} ,a_{n-1}^{-1} ,1^{n-2}\right) ,1,a_{n-2}^{-1} ,a_{n-3}^{-1} \dotsc ,a_{1}^{-1}\right)\right)\\
		& =g\left( a_{1}^{n-2} ,1,g\left( 1,1,a_{n-2}^{-1} ,\dotsc ,a_{1}^{-1}\right)\right)\\
		& =g\left( a_{1}^{n-3} ,1^{2} ,g\left( g\left( a_{n-2} ,a_{n-2}^{-1} ,1^{n-2}\right) ,1^{2} ,a_{n-3}^{-1} ,a_{n-4}^{-1} ,\dotsc ,a_{1}^{-1}\right)\right)\\
		& =g\left( a_{1}^{n-3} ,1^{2} ,g\left( 1,1^{2} ,a_{n-3}^{-1} ,a_{n-4}^{-1} ,\dotsc ,a_{1}^{-1}\right)\right)\\
		& \ \ \ \ \ \ \ \ \ \ \ \ \ \ \ \ \ \ \ \ \  \ \ \ \ \ \ \ \ \ \ \ \ \vdots \\
		& =g\left( a_{1} ,1^{n-2} ,g\left( 1,1^{n-2} ,a_{1}^{-1}\right)\right)\\
		&=g\left( 1^{n-1} ,g\left( 1,1^{n-1}\right)\right)\\
		& =1
	\end{align*}
and
\begin{align*}
	g\Bigl( g\bigl( a_{n}^{-1} , & a_{n-1}^{-1} ,\dotsc ,a_{1}^{-1}\bigr) ,g\left( a_{1}^{n}\right) ,1^{n-2}\Bigr)\\
	& =g\left( a_{n}^{-1} ,a_{n-1}^{-1} ,\dotsc ,a_{2}^{-1} ,g\left( g\left( a_{1}^{-1} ,a_{1} ,1^{n-2}\right) ,a_{2}^{n}\right)\right)\\
	& =g\left( a_{n}^{-1} ,a_{n-1}^{-1} ,\dotsc ,a_{2}^{-1} ,g\left( 1,a_{2}^{n}\right)\right)\\
	& =g\left( a_{n}^{-1} ,a_{n-1}^{-1} ,\dotsc ,a_{3}^{-1} ,1,g\left( g\left( a_{2}^{-1} ,a_{2} ,1^{n-2}\right) ,1,a_{3}^{n}\right)\right)\\
	& =g\left( a_{n}^{-1} ,a_{n-1}^{-1} ,\dotsc ,a_{3}^{-1} ,1,g\left( 1,1,a_{3}^{n}\right)\right)\\
	& =g\left( a_{n}^{-1} ,a_{n-1}^{-1} ,\dotsc ,a_{4}^{-1} ,1^{2} ,g\left( g\left( a_{3}^{-1} ,a_{3} ,1^{n-2}\right) ,1^{2} ,a_{4}^{n}\right)\right)\\
	& =g\left( a_{n}^{-1} ,a_{n-1}^{-1} ,\dotsc ,a_{4}^{-1} ,1^{2} ,g\left( 1,1^{2} ,a_{4}^{n}\right)\right)\\
	& \ \ \ \ \ \ \ \ \ \ \ \ \ \ \ \ \ \ \ \ \ \ \ \ \ \ \ \ \ \vdots \\
	& =g\left( a_{n}^{-1} ,1^{n-2} ,g\left( 1,1^{n-2} ,a_{n}\right)\right)\\
	&=g\left( 1^{n-1} ,1\right)\\
	& =1.
\end{align*}
Therefore, \( g\left( a_{n}^{-1} ,a_{n-1}^{-1} ,\dotsc ,a_{1}^{-1}\right)\) is \( g \)-inverse of \( g\left( a_{1}^{n}\right)\) and \( g\left( a_{1}^{n}\right)\) is invertible.
\end{proof}

\subsection{Homomorphism of \(\boldsymbol{(m,n)}\)-seminearrings} We remember the following definition adopted from \cite[Definition 16]{Alam}.
\begin{definition}
	A mapping \(\psi:R\to R'\) form \((m,n)\)-seminearring \((R,f,g)\) into \((m,n)\)-seminearring \((R',f',g')\) is called a homomorphism if
	\begin{align*}
		\psi\left( f\left(a_1^m \right)  \right) &=f'\left(\psi(a_1), \psi(a_2),\ldots,\psi(a_m) \right) \\
		\psi\left( g\left(b_1^n \right)  \right) &=g'\left(\psi(b_1), \psi(b_2),\ldots,\psi(b_n) \right) 
	\end{align*}
for all \(a_1^m,b_1^n\in R\).
\end{definition}
\newpage
\begin{definition}
	Let \((R,f,g)\) and \((R',f',g')\) be \((m,n)\)-seminearrings.
	\begin{enumerate}
		\item Homomorphism \(\psi :R\rightarrow R'\) is called a monomorphism if \(\psi\) is injective. 
		\item Homomorphism \(\psi :R\rightarrow R'\) is called an epimorphism if \(\psi\) is surjective.
		\item Homomorphism \(\psi :R\rightarrow R'\) is called an isomorphism if \(\psi\) is bijective.
	\end{enumerate}
\end{definition}
\begin{theorem}
	Let \((R,f,g)\), \((S,f',g')\), and \((T,f'',g'')\) be \((m,n)\)-seminearrings. If \( \psi :R\rightarrow S\) and \( \varphi :S\rightarrow T\) are homomorphisms, then \( \varphi \circ\psi :R\rightarrow T\) is also a homomorphism.
\end{theorem}
\begin{proof}
	Suppose that \(a_{1}^{m} ,b_{1}^{n} \in R\). Then
	\begin{align*}
		(\varphi \circ \psi)\left( f\left( a_{1}^{m}\right)\right) & =\varphi \left( \psi \left( f\left( a_{1}^{m}\right)\right)\right)\\
		& =\varphi( f'( \psi(a_{1}) ,\psi(a_{2}) ,\dotsc ,\psi(a_{m})))\\
		& =f''( \varphi( \psi(a_{1})) ,\varphi( \psi(a_{2})) ,\dotsc ,\varphi( \psi(a_{m})))\\
		&=f''(( \varphi \circ \psi )( a_{1}) ,\varphi \circ \psi )( a_{2}) ,\dotsc ,\varphi \circ \psi )( a_{m}))
	\end{align*}
	By a similar argument, \(( \varphi \circ \psi )\left( g\left( b_{1}^{n}\right)\right) =g''(( \varphi \circ \psi )( b_{1}) ,\dotsc ,( \varphi \circ \psi )( b_{n}))\). 	
\end{proof}
\begin{theorem}
	Let \((R,f,g)\) and \((R',f',g')\) be \((m,n)\)-seminearrings and \( \psi :R\rightarrow R'\) be function.
	\begin{enumerate}
		\item If \( \psi :R\rightarrow R'\) is a homomorphism, then \( \left(\emph{\text{Im}}( \psi ) ,f',g'\right)\) is an \((m,n)\)-subseminearring of \((R',f',g')\).
		\item If \( \psi :R\rightarrow R'\) is an epimorphism and \((I,f,g)\) is an \((m,n)\)-ideal of \((R,f,g)\), then \(( \psi( I),f',g')\) is an \((m,n)\)-ideal of \((R',f',g')\).
	\end{enumerate}
\end{theorem}
\begin{proof} (1) Since \( \psi\) is a function, \( \text{Im}( \psi )\neq \emptyset\). If \( a_{1}^{m} ,b_{1}^{n} \in \text{Im}( \psi )\), then there exists \( r_{1}^{m} , s_{1}^{n} \in R\) such that \( \psi(r_{i})=a_{i}\) \( (1\leqslant i\leqslant m)\) and \( \psi(s_{j})=b_{j}\) \( (1\leqslant j\leqslant n)\). Since
		\begin{equation*}
			f'\left( a_{1}^{m}\right)=f'( \psi(r_{1}) ,\psi(r_{2}) ,\dotsc ,\psi(r_{m})) =\psi \left( f\left( r_{1}^{m}\right)\right) \in \text{Im}( \psi )
		\end{equation*}
		and
		\begin{equation*}
			g'\left( b_{1}^{n}\right)=g'( \psi(s_{1}) ,\psi(s_{2}) ,\dotsc ,\psi(s_{n})) =\psi \left( g\left( s_{1}^{n}\right)\right) \in \text{Im}( \psi ) ,
		\end{equation*}
		by Lemma \ref{lemma-subseminearring}, \( \left(\text{Im}( \psi ) ,f',g'\right)\) is an \((m,n)\)-subseminearring of \((R',f',g')\).
		
		(2) Suppose that \(b_{1}^{i-1} ,b_{i+1}^{n} \in R'\) \((1\leqslant i\leqslant n)\) and \(a_{1}^{m} ,c\in \psi(I)\). Then there are \(y_{1}^{i-1} ,y_{i+1}^{n} \in R\) and \(x_{1}^{m} ,z\in I\) such that \(\psi (y_{j}) =b_{j} \ (1\leqslant j\leqslant n, \ j\neq i)\), \(\psi(x_{k})=a_{k}\) \((1\leqslant k\leqslant m)\), and \(\psi(z)=c\). Since \(x_{1}^{m}\in I\) and \(\displaystyle(I,f,g)\) is an \(\displaystyle(m,n)\)-ideal of \(\displaystyle(R,f,g)\), we have \(f\left( x_{1}^{m}\right) \in I\). Consequently, 
	\begin{equation*}
	f'\left( a_{1}^{m}\right)=f'( \psi( x_{1}) ,\psi( x_{2}) ,\dotsc ,\psi( x_{m})) =\psi \left( f\left( x_{1}^{m}\right) \right) \in \psi(I) .
\end{equation*}
Since \(y_{1}^{i-1} ,y_{i+1}^{n}\in R\), \(z\in I\), and \(\displaystyle(I, f , g)\) is an \(i\)-\((m,n)\)-ideal of \((R,f,g)\), we see that \(g\left( y_{1}^{i-1} ,z,y_{i+1}^{n}\right) \in I\). Therefore
\begin{align*}
	g'\left( b_{1}^{i-1} ,c,b_{i+1}^{n}\right)&=g'( \psi(y_{1}) ,\dotsc ,\psi(y_{i-1}) ,\psi(z) ,\psi(y_{i+1}) ,\dotsc ,\psi(y_{n})) \\
	&=\psi \left( g\left( y_{1}^{i-1} ,z,y_{i+1}^{n}\right)\right) \in \psi( I ).
\end{align*}
By Lemma \ref{lemma-ideals}, we conclude that \((\psi(I) ,f',g')\) is an \((m,n)\)-ideal of \((R',f',g')\).
\end{proof}
\subsection{Factor \(\boldsymbol{(m,n)}\)-seminearring} In this section, we discuss the construction of a new \( (m,n) \)-seminearring of \( (m,n) \)-seminearring by a congruence relation.
\begin{definition}
	Let \(( R,f,g)\) be an \(( m,n)\)-seminearring. An equivalence relation \( \rho \) on \(R\) is called an congruence relation or congruence provided that for any \(a_1^m,b_1^m,c_1^n,d_1^n\in R\):
	\begin{enumerate}
		\item If \( a_{i} \rho b_{i}\) \( ( 1\leqslant i\leqslant m)\) then \( f\left( a_{1}^{m}\right) \rho f\left( b_{1}^{m}\right)\);
		\item If \( c_{j} \rho d_{j}\) \( ( 1\leqslant j\leqslant n)\) then \( g\left( c_{1}^{n}\right) \rho g\left( d_{1}^{n}\right)\).
	\end{enumerate}
\end{definition}
\begin{theorem}
	Let \(( R_{1} ,f_{1} ,g_{1})\) and \(( R_{2} ,f_{2} ,g_{2})\) be \(( m,n)\)-seminearrings. If \(\psi :R_{1}\rightarrow R_{2}\) is a homomorphism then kernel of \(\psi\), denoted \(\text{\emph{Ker}}\left(  \psi\right) \), is defined \[
		\text{\emph{Ker}}\left( \psi\right) \coloneqq\left\{( a,b)\in R_{1}^{2}\middle|\psi(a)=\psi(b)\right\}
	\] is a congruence on \(R_{1}\).
\end{theorem}
\begin{proof}
	If \(a_{i} \text{Ker}\left(\psi\right) b_{i}\) \(( 1\leqslant i\leqslant m)\) and \(c_{j} \text{Ker}\left(\psi\right) d_{j}\) \(( 1\leqslant j\leqslant n)\), then \(\psi ( a_{i}) =\psi ( b_{i})\) \(( 1\leqslant i\leqslant m)\) and \(\psi ( c_{j}) =\psi ( d_{j})\) \(( 1\leqslant j\leqslant n)\), whence
	\begin{align*}
		\psi \left( f\left( a_{1}^{m}\right)\right) & =f'( \psi ( a_{1}) ,\psi ( a_{2}) ,\dotsc ,\psi ( a_{m}))\\
		& =f'( \psi ( b_{1}) ,\psi ( b_{2}) ,\dotsc ,\psi ( b_{m}))\\
		& =\psi \left( f\left( b_{1}^{m}\right)\right)
	\end{align*}
	and similarly \(\psi\left( g\left( c_{1}^{n}\right)\right) =\psi\left( g\left( d_{1}^{n}\right)\right)\). Therefore, \(f\left( a_{1}^{m}\right)\text{Ker}\left(\psi\right) f\left( b_{1}^{m}\right)\) and \(g\left( c_{1}^{n}\right)\) \(\text{Ker}\left(\psi\right) g\left( d_{1}^{n}\right)\).
\end{proof}
\begin{theorem}\label{faktor}
	Let \(( R,f,g)\) be an \((m,n)\)-seminearring and \(\rho\) be a congruence on \(R\). Then \(( R/\rho ,f',g')\) is together with \(m\)-ary operation \(f':( R/\rho )^{m}\rightarrow R/\rho \) defined by
	\begin{equation*}
		f'( \rho(x_{1}) ,\rho(x_{2}) ,\dotsc ,\rho(x_{m})) =\rho \left( f\left( x_{1}^{m}\right)\right)
	\end{equation*}
	and \(n\)-ary operation \(g':( R/\rho )^{n}\rightarrow R/\rho\) defined by
	\begin{equation*}
		g'( \rho(y_{1}) ,\rho(y_{2}) ,\dotsc ,\rho(y_{n})) =\rho \left( g\left( y_{1}^{n}\right)\right) .
	\end{equation*}
\end{theorem}

\begin{proof} Suppose that \(\rho(x_{i})=\rho(y_{i})\) \(( 1\leqslant i\leqslant m)\). Then \(x_{i} \rho y_{i}\) \(( 1\leqslant i\leqslant m)\). Since \(\rho\) is a congruence on \(R\) and \(x_{i} \rho y_{i}\) \(( 1\leqslant i\leqslant m)\), we have \(f\left( x_{1}^{m}\right) \rho f\left( y_{1}^{m}\right)\). Hence \(\rho \left( f\left( x_{1}^{m}\right)\right) =\rho \left( f\left( y_{1}^{m}\right)\right)\). Consequently, \[f'( \rho(x_{1}) ,\rho(x_{2}) ,\dotsc ,\rho(x_{m})) =f'( \rho( y_{1}) ,\rho( y_{2}) ,\dotsc ,\rho( y_{m})).\] Thus \(f'\) is well-defined. Analogously, we can show that \( g' \) is well-defined.
	
	Suppose that \(\rho(a_{1}) ,\rho(a_{2}) ,\dotsc ,\rho(a_{2m-1}) \in R/\rho\) and \(1\leqslant i< j\leqslant m\). Then
	\begin{align*}
		f'( \rho ( a_{1}) , & \dotsc ,\rho ( a_{i-1}) ,f'( \rho ( a_{i}),\dotsc ,\rho ( a_{m+i-1})) ,\rho ( a_{m+i}),\dotsc ,\rho ( a_{2m-1}))\\
		& =f'\left( \rho ( a_{1}) ,\dotsc ,\rho ( a_{i-1}) ,\rho \left( f\left( a_{i}^{m+i-1}\right)\right) ,\rho ( a_{m+i}),\dotsc ,\rho ( a_{2m-1})\right)\\
		& =\rho \left( f\left( a_{1}^{i-1} ,f\left( a_{i}^{m+i-1}\right) ,a_{m+i}^{2m-1}\right)\right)\\
		& =\rho \left( f\left( a_{1}^{j-1} ,f\left( a_{j}^{m+j-1}\right) ,a_{m+j}^{2m-1}\right)\right)\\
		& =f'\left( \rho ( a_{1}),\dotsc ,\rho ( a_{j-1}) ,\rho \left( f\left( a_{j}^{m+j-1}\right)\right) ,\rho ( a_{m+j}),\dotsc ,\rho ( a_{2m-1})\right)\\
		& =f'( \rho ( a_{1}),\dotsc ,\rho ( a_{j-1}) ,f'( \rho ( a_{j}),\dotsc ,\rho ( a_{m+j-1})) ,\rho ( a_{m+j}),\dotsc ,\rho ( a_{2m-1})),
	\end{align*}
	so \(( R/\rho ,f')\) is an \(m\)-semigroup. Similarly, \(( R/\rho ,g')\) is an \(n\)-semigroup. 
	
	Here, it will be proved that \(g'\) is left distributive with respect to \(f'\). Suppose that \(\rho(a_{1}) ,\rho(a_{2}) ,\dotsc ,\rho(a_{m}) ,\rho(b_{2}) ,\rho(b_{3}) ,\dotsc ,\rho(b_{n}) \in R/\rho\). Then
	\begin{align*}
		g'( f'( \rho ( a_{1}) , &\dotsc ,\rho ( a_{m})) ,\rho ( b_{2}),\dotsc ,\rho ( b_{n}))\\
		& =g'\left( \rho \left( f\left( a_{1}^{m}\right)\right) ,\rho ( b_{2}),\dotsc ,\rho ( b_{n})\right)\\
		& =\rho \left( g\left( f\left( a_{1}^{m}\right) ,b_{2}^{n}\right)\right)\\
		& =\rho \left( f\left( g\left( a_{1} ,b_{2}^{n}\right),\dotsc ,g\left( a_{m} ,b_{2}^{n}\right)\right)\right)\\
		& =f'\left( \rho \left( g\left( a_{1} ,b_{2}^{n}\right)\right),\dotsc ,\rho \left( g\left( a_{m} ,b_{2}^{n}\right)\right)\right)\\
		& =f'( g'( \rho ( a_{1}) ,\rho ( b_{2}),\dotsc ,\rho ( b_{n})),\dotsc ,g'( \rho ( a_{m}) ,\rho ( b_{2}) ,\dotsc ,\rho ( b_{n})))
	\end{align*}
	Hence, \(g'\) is left distributive with respect to \(f'\). Thus, we conclude 
	that \(( R/\rho ,f',g')\) is an \(( m,n)\)-seminearring.
\end{proof}
The new \((m,n)\)-seminearring structure discussed in Theorem \ref{faktor}, i.e. \((R/\rho, f', g')\), is called factor $(m,n)$-seminearring of \(R\) by \(\rho\).

\bibliographystyle{amsplain}

\end{document}